\documentclass{article}%
\usepackage{graphicx}
\usepackage{amsmath}
\usepackage{amsfonts}
\usepackage{amssymb}%
\providecommand{\U}[1]{\protect\rule{.1in}{.1in}}
\providecommand{\U}[1]{\protect\rule{.1in}{.1in}}
\providecommand{\U}[1]{\protect\rule{.1in}{.1in}}
\providecommand{\U}[1]{\protect\rule{.1in}{.1in}}
\providecommand{\U}[1]{\protect\rule{.1in}{.1in}}
\providecommand{\U}[1]{\protect\rule{.1in}{.1in}}
\providecommand{\U}[1]{\protect\rule{.1in}{.1in}}
\providecommand{\U}[1]{\protect\rule{.1in}{.1in}}
\providecommand{\U}[1]{\protect\rule{.1in}{.1in}}
\providecommand{\U}[1]{\protect\rule{.1in}{.1in}}
\providecommand{\U}[1]{\protect\rule{.1in}{.1in}}
\providecommand{\U}[1]{\protect\rule{.1in}{.1in}}
\providecommand{\U}[1]{\protect\rule{.1in}{.1in}}
\providecommand{\U}[1]{\protect\rule{.1in}{.1in}}
\providecommand{\U}[1]{\protect\rule{.1in}{.1in}}
\providecommand{\U}[1]{\protect\rule{.1in}{.1in}}
\providecommand{\U}[1]{\protect\rule{.1in}{.1in}}
\providecommand{\U}[1]{\protect\rule{.1in}{.1in}}
\providecommand{\U}[1]{\protect\rule{.1in}{.1in}}
\providecommand{\U}[1]{\protect\rule{.1in}{.1in}}
\providecommand{\U}[1]{\protect\rule{.1in}{.1in}}
\providecommand{\U}[1]{\protect\rule{.1in}{.1in}}
\providecommand{\U}[1]{\protect\rule{.1in}{.1in}}
\providecommand{\U}[1]{\protect\rule{.1in}{.1in}}
\providecommand{\U}[1]{\protect\rule{.1in}{.1in}}
\providecommand{\U}[1]{\protect\rule{.1in}{.1in}}
\providecommand{\U}[1]{\protect\rule{.1in}{.1in}}
\providecommand{\U}[1]{\protect\rule{.1in}{.1in}}
\providecommand{\U}[1]{\protect\rule{.1in}{.1in}}

\newtheorem{theorem}{Theorem}
{}

\newtheorem{definition}{Definition}

{}

\newtheorem{proposition}{Proposition}

\newenvironment{proof}[1][Proof]{\textbf{#1.} }{\ \rule{0.5em}{0.5em}}

\begin{document}

\title{On the Real Spectum of the One-Dimensional Dirac Operator with PT-Symmetric Coefficients}
\author{O. A. Veliev\\{\small \ Dogus University, }\\{\small Esenkent 34755, \ Istanbul, Turkey.}\\\ {\small e-mail: oveliev@dogus.edu.tr}}
\date{}
\maketitle

\begin{abstract}
In this paper, we consider the spectrum of the one dimensional
non-self-adjoint Dirac operator with $2\times2$ matrix-valued coefficients
whose entries are PT-symmetric periodic functions. We prove that the spectrum
of the considered operator contains a large part of the real axis.



\end{abstract}

\section{ Introduction and Preliminary Facts}

In this paper, we consider the spectrum of the one-dimensional Dirac operator
$L(P)$ generated in the space $L_{2}^{2}(-\infty,\infty)$ of the
complex-valued vector functions $\mathbf{y}(x)\mathbf{=}\left(
\begin{array}
[c]{c}%
y_{1}(x)\\
y_{2}(x)
\end{array}
\right)  $ by the differential expression
\begin{equation}
l(P)=\left(
\begin{array}
[c]{cc}%
-i & 0\\
0 & i
\end{array}
\right)  \mathbf{y}^{^{\prime}}(x)+P\mathbf{y}(x), \tag{1}%
\end{equation}
where $P(x)=\left(
\begin{array}
[c]{cc}%
p_{1}(x) & p_{2}(x)\\
p_{3}(x) & p_{4}(x)
\end{array}
\right)  ,$ $p_{j}$ for $j=1,2,3,4$ are $\pi$-periodic, complex-valued,
PT-symmetric, locally integrable functions.

It is well-known [19, 20, 23, 31,40] that the spectrum $\sigma(L(P))$ of the
operator $L(P)$ is the union of the spectra $\sigma(L_{t}(P))$ of the
operators $L_{t}(P)$ for $t\in\lbrack0,2)$ generated in $L_{2}^{2}[0,\pi]$ by
(1) and the quasiperiodic boundary condition
\begin{equation}
\mathbf{y}(\pi)=e^{i\pi t}\mathbf{y}(0). \tag{2}%
\end{equation}
The spectrum of $L_{t}(P)$ consists of eigenvalues, called the Bloch
eigenvalues of $L(P).$

Let us briefly describe the structure of this paper. First, we consider the
more general differential operator $L(n,m)$ generated in the space $L_{2}%
^{m}(-\infty,\infty)$ by the differential expression
\begin{equation}
l(n,m)=(-i)^{n}P_{0}\mathbf{y}^{(n)}+(-i)^{n-1}P_{1}\mathbf{y}^{(n-1)}%
+(-i)^{n-2}P_{2}\mathbf{y}^{(n-2)}+...+P_{n}\mathbf{y}, \tag{3}%
\end{equation}
where $n\geq1,$ $P_{k}=(p_{k,i,j})$ for $k=0,1,...,n$ are the $m\times m$
matrices with the complex-valued entries $p_{k,i,j}$ satisfying the following
conditions
\begin{equation}
p_{k,i,j}\in W_{1}^{n-k}[0,\pi],\text{ }p_{k,i,j}\left(  x+\pi\right)
=p_{k,i,j}\left(  x\right)  ,\text{ }p_{k,i,j}\left(  -x\right)
=\overline{p_{k,i,j}\left(  x\right)  },\text{ }\det P_{0}(x)\neq0 \tag{4}%
\end{equation}
for all $x\in\mathbb{R}.$ Here $\mathbf{y}=(y_{1},y_{2},...,y_{m})^{T}$ is a
vector-valued function, $L_{2}^{m}(a,b)$ for $-\infty\leq a<b\leq\infty$ is
the space of the vector-valued functions $\mathbf{f}=\left(  f_{1}%
,f_{2},...,f_{m}\right)  ^{T}$ with the norm $\left\Vert \cdot\right\Vert $
and inner product $(\cdot,\cdot)$ defined by%
\[
\left\Vert \mathbf{f}\right\Vert =\left(  \int\nolimits_{(a,b)}\left\vert
\mathbf{f}\left(  x\right)  \right\vert ^{2}dx\right)  ^{\frac{1}{2}},\text{
}(\mathbf{f,g})=\int\nolimits_{(a,b)}\left\langle \mathbf{f}\left(  x\right)
,\mathbf{g}\left(  x\right)  \right\rangle dx,
\]
where $\left\vert \cdot\right\vert $ and $\left\langle \cdot,\cdot
\right\rangle $ are the norm and inner product in $\mathbb{C}^{m}.$

We use the following definition of a PT-symmetric expression and a
PT-symmetric operator.

\begin{definition}
We say that the expression $l(n,m)$ and the operator $L(n,m)$ are PT-symmetric
if
\[
PTl(n,m)\mathbf{y}=l(n,m)PT\mathbf{y}%
\]
for all $\mathbf{y}$ in the domain $D(L(n,m))$ of the operator $L(n,m),$ where
the space-reflection (parity) operator $P$ and the complex-conjugation
operator $T$ are defined by $(P\mathbf{y})(x)=\mathbf{y}(-x)$ and
$(T\mathbf{y})(x)=\overline{\mathbf{y}(x)}.$
\end{definition}

Note that the definition of PT-symmetry for $l(n,m)$ and $L(n,m)$ are
equivalent, since the operator $L(n,m)$ is defined by $l(n,m)$ without
boundary conditions. For operators defined on bounded intervals, however, the
boundary conditions must also be taken into account (see Definition 2 in
Section 3).

The main results of this paper are as follows and are proved in Section 2.
First, we prove that the expression $l(n,m)$ and the operator $L(n,m)$ are
PT-symmetric if and only if all the functions $p_{k,i,j}$ are PT-symmetric
(see Theorem 1). It follows from Theorem 1 that the Dirac operator $L(P)$ is
PT-symmetric operator if $p_{j}$ for $j=1,2,3,4$ are the PT-symmetric
functions, since, in case (1), $P_{0}$ has the form $\left(
\begin{array}
[c]{cc}%
1 & 0\\
0 & -1
\end{array}
\right)  $ and the constant functions $f=\pm1$ are PT-symmetric. Note that
there are two alternative forms of writing the Dirac system. In this paper, we
consider the system of the form (1). Another form of the Dirac system (see,
for example, [25]) is%
\begin{equation}
\left(
\begin{array}
[c]{cc}%
0 & -1\\
1 & 0
\end{array}
\right)  \mathbf{y}^{^{\prime}}(x)+\left(
\begin{array}
[c]{cc}%
p_{1}(x) & p_{2}(x)\\
p_{3}(x) & p_{4}(x)
\end{array}
\right)  \mathbf{y}(x). \tag{5}%
\end{equation}
Dirac operator defined by (5) is not PT-symmetric, since, in this case,
$P_{0}$ has the form $\frac{1}{-i}\left(
\begin{array}
[c]{cc}%
0 & -1\\
1 & 0
\end{array}
\right)  =\left(
\begin{array}
[c]{cc}%
0 & -i\\
i & 0
\end{array}
\right)  $ and the constant functions $f=\pm i$ are not PT-symmetric. \ 

Since a fundamental mathematical question in PT-symmetric quantum mechanics
concerns the reality of the spectrum of the non-self-adjoint operator under
consideration (see [1, 3, 4, 6, 7, 21, 22, 28, 30, 32-39] and the references
therein), we focus on the real part of the spectrum of $L(P)$. In other words,
we consider the real Bloch eigenvalues in the PT-symmetric case. In fact, in
this case, a characteristic property of the Bloch eigenvalues is that they are
either real or occur in complex-conjugate pairs. In the case of the
Schr\"{o}dinger operator $S_{t},$ generated in $L_{2}[0,\pi]$ by the
expression $-y^{^{\prime\prime}}(x)+q\left(  x\right)  y(x),$ with
PT-symmetric potential $q$ and quasiperiodic boundary conditions, it is
well-known that if $\lambda$ is an eigenvalue of $S_{t}$, then $\overline
{\lambda}$ is also an eigenvalue of $S_{t}$ (see [18, 32]). Taking into
account that this fact has not been formulated for $L_{t}(n,m)$ with arbitrary
$n$ and $m$, in Section 2, we prove this result for the general case (see
Theorem 2), where $L_{t}(n,m)$ is the operator generated in $L_{2}^{m}[0,\pi]$
by the differential expression (3) and boundary conditions
\begin{equation}
\mathbf{y}(\pi)=e^{i\pi t}\mathbf{y}(0),\text{ }\mathbf{y}^{\prime}%
(\pi)=e^{i\pi t}\mathbf{y}^{\prime}(0),...,\mathbf{y}^{(n-1)}(\pi)=e^{i\pi
t}\mathbf{y}^{(n-1)}(0).\tag{6}%
\end{equation}
Then, using this result together with the asymptotic estimates for the
eigenvalues of the operator $L_{t}(P)$ (see Theorem 3), we prove that the real
part of the spectrum of $L(P)$ contains a large portion of $(-\infty,\infty)$
(see Theorem 4).

The case in which $L(n,m)$ is generated by the expression
\[
-y^{^{\prime\prime}}(x)+Q\left(  x\right)  y(x),
\]
where the entries of the $m\times m$ matrix $Q$ are PT-symmetric functions, as
well as the case of higher-order differential expressions with PT-symmetric
matrix coefficients, were considered in [35] and [36-38 ]; see also Chapters 6
and 7 of [39]. Note that the perturbation theory used in those papers for
$n\geq2$ cannot be applied to $L(P)$ for the following reason. The Bloch
eigenvalues of operator of order $n\geq2$ are located in neighborhoods of
$(2k+t)^{n},$ for $k\in\mathbb{Z}$, where the distance between neighboring
neighborhoods tends to infinity as $k\rightarrow\infty.$ This property of the
distances is an essential ingredient in [35-39] for the case $n\geq2$, but it
cannot be used in the case $n=1$ of the Dirac operator $L(P)$.

Therefore, for the asymptotic estimates of the Bloch eigenvalues of the Dirac
operator, we use the asymptotic formulas for the matrix
\[
E(x,\lambda)=\left(
\begin{array}
[c]{cc}%
e_{11}(x,\lambda) & e_{12}(x,\lambda)\\
e_{21}(x,\lambda) & e_{22}(x,\lambda)
\end{array}
\right)
\]
of the fundamental system of solutions for the equation
\begin{equation}
\left(
\begin{array}
[c]{cc}%
-i & 0\\
0 & i
\end{array}
\right)  \mathbf{y}^{^{\prime}}(x)+\left(
\begin{array}
[c]{cc}%
0 & \widetilde{p_{2}}(x)\\
\widetilde{p_{3}}(x) & 0
\end{array}
\right)  \mathbf{y}(x)=\lambda\mathbf{y}(x), \tag{7}%
\end{equation}
with the initial conditions $E(0,\lambda)=I,$ obtained in [25], where it was
proved that
\begin{align}
e_{11}(x,\lambda)  &  =e^{i\lambda x}+\rho_{11}(x,\lambda),\text{ }%
e_{12}(x,\lambda)=\rho_{11}(x,\lambda),\text{ }\tag{8}\\
e_{21}(x,\lambda)  &  =\rho_{21}(x,\lambda),\text{ }e_{22}(x,\lambda
)=e^{-i\lambda x}+\rho_{22}(x,\lambda)\nonumber
\end{align}
and $\rho_{kl}(x,\lambda)\rightarrow0,$ as $\lambda\rightarrow\infty$ \ for
$k,l\in\left\{  1,2\right\}  .$ Moreover, these estimates hold uniformly with
respect to $x\in\lbrack0,\pi]$ and $\left\vert \operatorname{Im}%
\lambda\right\vert <a,$ where $a$ is a fixed number. First, we consider the
operator $L\left(  \widetilde{P}\right)  ,$ where
\begin{equation}
\widetilde{P}(x)=\left(
\begin{array}
[c]{cc}%
0 & \widetilde{p_{2}}(x)\\
\widetilde{p_{3}}(x) & 0
\end{array}
\right)  ,\overline{\widetilde{p_{2}}(-x)}=\widetilde{p_{2}}(x),\text{
}\overline{\widetilde{p_{3}}(-x)}=\widetilde{p_{3}}(x). \tag{9}%
\end{equation}
Then, using this and Statement 1.3 of [25], about reduces case (1) to case
(7), we consider the operator $L(P).$

Note that there are many papers on Dirac operators in $L_{2}^{2}[a,b],$ for
$-\infty<a<b<\infty,$ under various boundary conditions (see [5, 8-17, 24-27,
29] and the references therein). Here, we use only the asymptotic formulas
(8), obtained in [25], for the solutions, which do not depend on the boundary
conditions. These asymptotic formulas are used in Section 2 to obtain estimate
for the eigenvalues of the family of operators $\left\{  L_{t}(P):t\neq
0,1\right\}  $ that is uniform with respect to $t$. Certainly, asymptotic
formulas for the eigenvalues of $L(P,U)$ are available for different boundary
conditions $U(\mathbf{y})=0$. However, it is easier to obtain an asymptotic
formula for this family of operators that is uniform with respect to $t$ by
using the asymptotic formulas (8) and Rouch\'{e}'s theorem than to establish
the uniformity of the asymptotic formulas obtained in different papers for
different boundary conditions. Therefore, we do not discuss in detail the
works devoted to Dirac operators on a bounded interval.

\section{Main Results}

First, let us consider the general operator $L(n,m),$ using Definition 1.

\begin{theorem}
The expression $l(n,m)$ and the operator $L(n,m)$ are PT-symmetric if and only
if $p_{k,i,j}$ are PT-symmetric functions for all $k,i$, $j.$
\end{theorem}

\begin{proof}
\bigskip Let $p_{k,i,j}$ be PT-symmetric functions for all $k,i$, $j$. The
operator $L(n,m)$ can be written in the form
\[
L(n,m)\mathbf{y}=\sum_{k=0}^{n}P_{k}D^{n-k}\mathbf{y},
\]
where $D\mathbf{y}=-i\mathbf{y}^{^{\prime}}.$ It is enough to prove that
\begin{equation}
\left(  P_{k}D^{n-k}\right)  PT\mathbf{y}(x)=PT\left(  P_{k}D^{n-k}%
\mathbf{y}(x)\right)  \tag{10}%
\end{equation}
for each $k=0,1,...,n$ and for all $\mathbf{y}\in D(L(n,m)).$ To this end, let
us calculate the left- and right-hand sides of (10) separately. Since
$PT\mathbf{y}(x)=\overline{\mathbf{y}(-x)}$ , by the chain rule, we have
$\left(  PT\mathbf{y}(x)\right)  ^{^{\prime}}$ $=\left(  \overline
{\mathbf{y}(-x)}\right)  ^{^{\prime}}$ $=-\overline{\mathbf{y}^{\prime}(-x)}$
and $D\overline{\mathbf{y}(-x)}=i\left(  \overline{\mathbf{y}^{\prime}%
(-x)}\right)  .$ Hence,
\[
\text{ }\left(  P_{k}(x)D^{n-k}\right)  PT\mathbf{y}(x)=i^{n-k}P_{k}%
(x)\overline{\mathbf{y}^{(n-k)}(-x)}.
\]
On the other hand, for the right-hand side of (10), we have%
\begin{align*}
PT\left(  P_{k}(x)D^{n-k}\mathbf{y}(x)\right)   &  =PT\left(  P_{k}%
(x)(-i)^{n-k}\mathbf{y}^{(n-k)}(x)\right)  =\\
P(\overline{P_{k}(x)}(i)^{n-k}\overline{\mathbf{y}^{(n-k)}(x)})  &
=(\overline{P_{k}(-x)}(i)^{n-k}\overline{\mathbf{y}^{(n-k)}(-x)}).
\end{align*}
Therefore, (10) holds due to the equality $\overline{P_{k}(-x)}=P_{k}(x).$
Thus, the sufficiency is proved.

Now suppose that $\overline{P_{k}(-x)}-P_{k}(x)\neq O$ for $k\in\left\{
k_{\,1},k_{2},...,k_{s}\right\}  ,$ but $L(n,m)$ is a PT-symmetric operator,
where $0\leq k_{\,1}<k_{2}<...<k_{s}\leq n.$ Then, by Definition 1, we have
\[
PTL(n,m)\mathbf{y}(x)-L(n,m)PT\mathbf{y}(x)\mathbf{=}\sum_{j=1}^{s}%
(i)^{n-k_{j}}\left(  (\overline{P_{k_{j}}(-x)}-P_{k_{j}}(x))\overline
{\mathbf{y}^{n-k_{j}}(-x)}\right)  .
\]
Therefore, the equality
\begin{equation}
PTL(n,m)\mathbf{y}(x)=L(n,m)PT\mathbf{y}(x) \tag{11}%
\end{equation}
is satisfied only for those functions $\mathbf{y}(x)$ for which $\overline
{\mathbf{y}(-x)}$ is a solution of the equation
\[
\sum_{j=1}^{s}(i)^{n-k_{j}}\left(  (\overline{P_{k_{j}}(-x)}-P_{k_{j}%
}(x))\overline{\mathbf{y}^{n-k_{j}}(-x)}\right)  =0.
\]
Since the set of solutions of this equation is a finite dimensional subset $E$
of $D(L(n,m))$, (11) is not satisfied for any $\mathbf{y}\in\left(
D(L(n,m))\backslash E\right)  .$ This means that $L(n,m)$ is not a
PT-symmetric operator. The theorem is proved.
\end{proof}

Now, using Theorem 1 and Definition 1, we prove the following property of the
eigenvalues of the operator $L_{t}(n,m)$ for arbitrary $n$ and $m$.

\begin{theorem}
If $\lambda$ is an eigenvalue of multiplicity $v$ of the operator
$L_{t}(n,m),$ and (4) holds, then $\overline{\lambda}$ is also an eigenvalue
of multiplicity $v$ of $L_{t}(n,m).$
\end{theorem}

\begin{proof}
Let $\mathbf{\Psi}(\cdot,\lambda)$ be an eigenfunction corresponding to the
eigenvalue $\lambda.$ This means that $\mathbf{\Psi}(\cdot,\lambda)$ is a
solution equation of
\begin{equation}
l(n,m)\mathbf{\Psi}(\cdot,\lambda)=\lambda\mathbf{\Psi}(\cdot,\lambda)
\tag{12}%
\end{equation}
satisfying (6). Then $\mathbf{\Psi}(x+\pi,\lambda)$ is also a solution of
(12). Moreover $\mathbf{y}(x,\lambda):=\mathbf{\Psi}(x+\pi,\lambda
)-e^{it}\mathbf{\Psi}(x,\lambda)$ is a solution of (12) satisfying the initial
conditions
\[
\mathbf{y}(0)=\mathbf{y}^{\prime}(0)=...=\mathbf{y}^{(n-1)}(0)=0.
\]
Then, by the uniqueness theorem, we obtain
\begin{equation}
\mathbf{\Psi}(x+\pi,\lambda)=e^{it}\mathbf{\Psi}(x,\lambda) \tag{13}%
\end{equation}
for all $x\in(-\infty,\infty).$ Since $PT$ is an invertible operator, it
follows from (12) that
\[
PTl(n,m)\mathbf{\Psi}(\cdot,\lambda)=\overline{\lambda}PT\mathbf{\Psi}%
(\cdot,\lambda)
\]
Using Theorem 1 and Definition 1, we conclude that
\begin{equation}
l(n,m)PT\mathbf{\Psi}(\cdot,\lambda)=\overline{\lambda}PT\mathbf{\Psi}%
(\cdot,\lambda), \tag{14}%
\end{equation}
where $PT\mathbf{\Psi}(x,\lambda)=\overline{\mathbf{\Psi}(-x,\lambda
)}=:\mathbf{\Phi}(x,\lambda).$ Moreover, it follows from (13) that
\[
\mathbf{\Phi}(x+\pi,\lambda)=\overline{\mathbf{\Psi}(-x-\pi,\lambda
)}=\overline{e^{-it}\mathbf{\Psi}(-x,\lambda)}=e^{it}\mathbf{\Phi}(x).
\]
Thus, $PT\mathbf{\Psi}(\cdot,\lambda)$ is an eigenfunction of $\ L_{t}(n,m)$
corresponding to the eigenvalue $\overline{\lambda}.$

Instead of equation (12), using the equation
\[
\left(  l(n,m)-\lambda I\right)  \mathbf{\Psi}_{1}(\cdot,\lambda
)=\mathbf{\Psi}(\cdot,\lambda)
\]
for the associated function $\mathbf{\Psi}_{1}(x,\lambda)$ and repeating the
above argument, we conclude that if $\mathbf{\Psi}_{1}(x,\lambda)$ is an
associated function corresponding to the eigenfunction $\mathbf{\Psi
}(x,\lambda),$ then $\overline{\mathbf{\Psi}_{1}(-x,\lambda)}$ \ is an
associated eigenfunction corresponding to the eigenfunction $\overline
{\mathbf{\Psi}(-x,\lambda)}.$ Repeating this process, we obtain that the
multiplicities of the eigenvalues $\lambda$ and $\overline{\lambda}$ are the
same. The theorem is proved.
\end{proof}

Now, using Theorems 1 and 2, we prove the main result of this paper for the
operator $L\left(  \widetilde{P}\right)  $. First, we consider asymptotic
estimates for the eigenvalues of $L_{t}\left(  \widetilde{P}\right)  \ $that
are uniform with respect to $t$. A number $\lambda$ is an eigenvalue of
$L_{t}\left(  \widetilde{P}\right)  $ if and only if it is a root of the
characteristic equation
\begin{equation}
\Delta(\lambda,t)=\left\vert
\begin{array}
[c]{cc}%
e_{11}(\pi,\lambda)-e^{i\pi t} & e_{21}(\pi,\lambda)\\
e_{12}(\pi,\lambda) & e_{22}(\pi,\lambda)-e^{i\pi t}%
\end{array}
\right\vert =0. \tag{15}%
\end{equation}
Using (8), we obtain the equality
\begin{equation}
\Delta(\lambda,t)=e^{i2\pi t}-e^{i\pi t}2\cos\pi\lambda+1+o(1)=0, \tag{16}%
\end{equation}
as $\lambda\rightarrow\infty.$ Thus, the eigenvalues $\lambda(t)$ of
$L_{t}\left(  \widetilde{P}\right)  $ lying in the strip $\left\vert
\operatorname{Im}\lambda\right\vert <a$ are the roots of the equation%
\begin{equation}
g(\lambda,t):=\cos\pi\lambda-\cos\pi t+e^{-i\pi t}f(\lambda)=0, \tag{17}%
\end{equation}
where $f(\lambda)$ is an entire function satisfying $f(\lambda)\rightarrow0,$
as $\lambda\rightarrow\infty$ and $\left\vert \operatorname{Im}\lambda
\right\vert <a.$

Let $\left\{  \varepsilon_{k}:k=1,2,...\right\}  $ and $\left\{  \delta
_{k}:k=1,2,...\right\}  $ be sequences such that $\varepsilon_{k}\rightarrow0$
and $\delta_{k}\rightarrow0$ as $k\rightarrow\infty,$ and
\begin{equation}
\frac{9}{2}\varepsilon_{k}^{2}>\delta_{k},\text{ }\left\vert f(\lambda
)\right\vert <\delta_{k}, \tag{18}%
\end{equation}
for $\lambda\in D_{k}:=\left\{  \lambda\in\mathbb{C}:\operatorname{Re}%
\lambda\in\lbrack2k,2k+2],\text{ }\left\vert \operatorname{Im}\lambda
\right\vert \leq\varepsilon_{k}\right\}  $ and $\left\vert k\right\vert >N,$
where $N$ is a sufficiently large natural number. Using (17) and Rouch\'{e}'s
theorem, we prove that there exist positive integer $N$ such that equations
(17) and%
\begin{equation}
\cos\pi\lambda-\cos\pi t=0 \tag{19}%
\end{equation}
\ have the same number of zeros inside each of the circles
\[
C(2k\pm t,\varepsilon_{k})=\left\{  \mu\in\mathbb{C}:\left\vert \mu-(2k\pm
t)\right\vert =\varepsilon_{k}\right\}
\]
for all $\left\vert k\right\vert >N$ and $t\in T(k):=\left(  [\varepsilon
_{k},1-\varepsilon_{k}]\cup\lbrack1+\varepsilon_{k},2-\varepsilon_{k}]\right)
.$ For this purpose, we prove that
\begin{equation}
\left\vert f(\lambda)\right\vert <\left\vert \cos\pi\lambda-\cos\pi
t\right\vert \tag{20}%
\end{equation}
for all%
\begin{equation}
\lambda\in C(2k\pm t,\varepsilon_{k}),\text{ }t\in T(k). \tag{21}%
\end{equation}
Using the trigonometric identity
\[
\left\vert \cos\pi\lambda-\cos\pi(2k\pm t)\right\vert =2\left\vert \sin
\frac{\pi}{2}(\lambda-(2k\pm t))\right\vert \left\vert \sin\frac{\pi}%
{2}(\lambda+(2k\pm t))\right\vert
\]
and the representation $\lambda=2k\pm t+e^{i\alpha}\varepsilon_{k}$,
$\alpha\in\lbrack0,2\pi)$ of $\lambda\in C(2k\pm t,\varepsilon_{k}),$ we
obtain
\begin{equation}
\left\vert \cos\pi\lambda-\cos\pi(2k\pm t)\right\vert =2\left\vert \sin
\frac{\pi}{2}e^{i\alpha}\varepsilon_{k}\right\vert \left\vert \sin(\frac{\pi
}{2}e^{i\alpha}\varepsilon_{k}\pm\pi t)\right\vert . \tag{22}%
\end{equation}
Now, let us estimate the right side of (22) for $t\in T(k).$ Using the
Maclaurin's series for $\sin z,$ we obtain
\begin{equation}
\sin\frac{\pi}{2}e^{i\alpha}\varepsilon_{k}=\frac{\pi}{2}e^{i\alpha
}\varepsilon_{k}+O(\varepsilon_{k}^{3}),\text{ }\left\vert \sin\frac{\pi}%
{2}e^{i\alpha}\varepsilon_{k}\right\vert >\frac{3}{2}\varepsilon_{k}. \tag{23}%
\end{equation}
If $t\in\lbrack\varepsilon_{k},1-\varepsilon_{k}],$ then $\pi t\in\lbrack
\pi\varepsilon_{k},\pi-\pi\varepsilon_{k}].$ Therefore, using the Taylor
expansion of $\sin z$ about a point $t_{0}\in\lbrack\pi\varepsilon_{k},\pi
-\pi\varepsilon_{k}]$, we conclude that
\begin{equation}
\left\vert \sin(\frac{\pi}{2}e^{i\alpha}\varepsilon_{k}\pm\pi t)\right\vert
>\left\vert \sin t_{0}\right\vert -\left\vert \frac{\pi}{2}e^{i\alpha
}\varepsilon_{k}\right\vert -\left\vert O(\varepsilon_{k}^{2})\right\vert
>\frac{3}{2}\varepsilon_{k}. \tag{24}%
\end{equation}
In the same way, we prove (24) for all $t\in T(k).$ It follows from (22)-(24)
\[
\left\vert \cos\pi\lambda-\cos\pi t\right\vert >\frac{9}{2}\varepsilon_{k}%
^{2}.
\]
if (21) holds. Therefore, (20) follows from (18). Since equation (19) has a
unique root inside circle $C(2k\pm t,\varepsilon_{k})$, Rouch\'{e}'s theorem
implies that equation (17) also has a unique root inside this circle. Thus, we
have proved the following:

\begin{theorem}
There exists a positive constant $N$ such that the operator $L_{t}\left(
\widetilde{P}\right)  $ has a unique eigenvalue ,counting multiplicities,
inside the circle $C(2k\pm t,\varepsilon_{k})$, for $t\in T(k)$ and
$\left\vert k\right\vert >N.$
\end{theorem}

Now, using this theorem, we consider the spectrum of PT-symmetric operator
$L\left(  \widetilde{P}\right)  .$

\begin{theorem}
For $\left\vert k\right\vert >N$ the intervals $[2k+2\varepsilon
_{k},2k+1-2\varepsilon_{k}]$ and $[2k+1+2\varepsilon_{k},2k+2-2\varepsilon
_{k}]$ belong to the spectrum of the operator $L\left(  \widetilde{P}\right)
,$ where $N$ is defined in Theorem 3. Consequently, the main part of the real
axis belongs to the spectrum $\sigma(L\left(  \widetilde{P}\right)  )$ of
$L\left(  \widetilde{P}\right)  $ in the sense that
\begin{equation}
\lim_{s\rightarrow\infty}\frac{\mu(\sigma(L\left(  \widetilde{P}\right)
)\cap(-2s,2s))}{4s}=1. \tag{25}%
\end{equation}

\end{theorem}

\begin{proof}
By Theorem 3, for $t\in T(k)$ and $\left\vert k\right\vert >N,$ there exists
only one eigenvalue of the operator $L_{t}\left(  \widetilde{P}\right)  $
inside the circle $C(2k+t,\varepsilon_{k}).$ Denote this eigenvalue by
$\lambda_{k}(t).$ If $\operatorname{Im}\lambda_{k}(t)\neq0,$ then, by Theorem
2, the eigenvalue $\overline{\lambda_{k}(t)}$ also lies inside the circle
$C(2k+t,\varepsilon_{k}).$ Then this circle contains at least two eigenvalue
of $L_{t}\left(  \widetilde{P}\right)  ,$ which contradicts Theorem 3. Thus
$\lambda_{k}(t)$ is areal number.

Now, we prove that $\lambda_{k}(t)$ depends continuously on $t$ in the
intervals $(\varepsilon_{k},1-\varepsilon_{k})$ and $(1+\varepsilon
_{k},2-\varepsilon_{k}).$ Take any point $t_{0}$ in one of these intervals,
say, $t_{0}\in(\varepsilon_{k},1-\varepsilon_{k}).$ Since $\lambda_{k}(t_{0})$
is a simple eigenvalue, it is a simple root of the equation (17). Thus,
$\frac{\partial g(\lambda,t)}{\partial\lambda}$ does not vanish at the point
$\left(  \lambda_{k}(t_{0}),t_{0}\right)  .$ Therefore, by the implicit
function theorem, there exists a sufficiently small number $\gamma_{k}$ and a
function $\lambda(t)$ such that $\left(  t_{0}-\gamma_{k},t_{0}+\gamma
_{k}\right)  \subset(\varepsilon_{k},1-\varepsilon_{k}),$ $\lambda(t)$ is
holomorphic function on $\left(  t_{0}-\gamma_{k},t_{0}+\gamma_{k}\right)  $
and $\lambda(t_{0})=\lambda_{k}(t_{0}).$ On the other hand, the centers of the
circles $C(2k+t,\varepsilon_{k})$ depend continuously on $t.$ Therefore, there
exist a neighborhood of $t_{0}$ in which $\lambda(t)$ coincides with
$\lambda_{k}(t)$. This implies that $\lambda_{k}(t)$ is continuous at $t_{0}.$
Hence $\lambda_{k}((\varepsilon_{k},1-\varepsilon_{k}))$ is an interval
contained in the spectrum of $L\left(  \widetilde{P}\right)  $. Since
$\lambda_{k}(\varepsilon_{k})$ and $\lambda_{k}(1-\varepsilon_{k})$ lie,
respectively, inside $C(2k+\varepsilon_{k},\varepsilon_{k})$ and
$C(2k+1-\varepsilon_{k},\varepsilon_{k}),$ we have
\[
\lambda_{k}(\varepsilon_{k})<2k+2\varepsilon_{k}\text{ and }\lambda
_{k}(1-\varepsilon_{k})>2k+1-2\varepsilon_{k}.\text{ }%
\]
Therefore, the interval $[2k+2\varepsilon_{k},2k+1-2\varepsilon_{k}]$ belongs
to the spectrum of the operator $L\left(  \widetilde{P}\right)  $ for all
$\left\vert k\right\vert >$ $N$. In the same way, we prove that
$[2k+1+2\varepsilon_{k},2k+2-2\varepsilon_{k}]\subset\sigma(L\left(
\widetilde{P}\right)  ).$ Thus, (25) holds, and the theorem is proved.
\end{proof}

Now, we consider $L(P)$ by using the following statement from [25].

\textit{Statement 1.3. (Savchuk, A.M., Sadovnichaya, I.V.) Let }$P(x)=\left(
\begin{array}
[c]{cc}%
p_{1}(x) & p_{2}(x)\\
p_{3}(x) & p_{4}(x)
\end{array}
\right)  $\textit{ be an arbitrary }$2\times2$\textit{ matrix with entries
}$p_{j}$\textit{ }$\in L_{1}[0,\pi]$ for $j=1,2,3,4$\textit{, and
}$C\mathbf{y}(0)+D\mathbf{y}(\pi)=0$\textit{ be the regular boundary
conditions, where }$C$\textit{ and }$D$ are $2\times2$ matrices.\textit{ Then
the operator }$L_{P,C,D}$\textit{ defined by differential expression (1) and
the above boundary condition is similar to the operator }$L_{\widetilde
{P},C,\widetilde{D}}+\gamma I$\textit{, where }$L_{\widetilde{P}%
,C,\widetilde{D}}$ \textit{is generated by the expression}
\[
l\left(  \widetilde{P}\right)  =\left(
\begin{array}
[c]{cc}%
-i & 0\\
0 & i
\end{array}
\right)  \mathbf{y}^{^{\prime}}(x)+\left(
\begin{array}
[c]{cc}%
0 & \widetilde{p_{2}}(x)\\
\widetilde{p_{3}}(x) & 0
\end{array}
\right)  \mathbf{y}(x)
\]
\textit{and boundary condition} $C\mathbf{y}(0)+\widetilde{D}\mathbf{y}%
(\pi)=0$\textit{.} \textit{Here}%
\[
\widetilde{p}_{2}(x)=p_{2}(x)e^{i(\varphi(x)-\psi(x))},\text{ }\widetilde
{p}_{3}(x)=p_{3}(x)e^{-i(\varphi(x)-\psi(x))},\text{ }\varphi(x)=\gamma
x-\int_{0}^{x}p_{1}(t)dt,
\]

\textit{and}
\[
\psi(x)=\int_{0}^{x}p_{4}(t)dt-\gamma x,\text{ \ }\gamma=\frac{1}{2\pi}%
\int_{0}^{\pi}\left(  p_{1}(t)+p_{4}(t)\right)  dt,\widetilde{\text{ }%
D}=e^{\frac{i}{2}\int_{0}^{\pi}\left(  p_{1}(t)-p_{4}(t)\right)  dt}D.
\]

In the case of boundary condition (2), we have $C=e^{i\pi t}I,$ and
$D=-I,$where $I$ is the unit matrix. Then $\widetilde{D}=e^{i\pi\tau}I,$
where
\begin{equation}
\tau=\frac{1}{2\pi}\int_{0}^{\pi}\left(  p_{1}(t)-p_{4}(t)\right)  dt.
\tag{26}%
\end{equation}
Therefore $L_{\widetilde{P},C,\widetilde{D}}$ is the operator generated by the
expression $l\left(  \widetilde{P}\right)  $ and the boundary condition
\begin{equation}
y(\pi)=e^{i\pi(t-\tau)}y(0). \tag{27}%
\end{equation}
In other words, $L_{\widetilde{P},C,\widetilde{D}}$ is $L_{t-\tau}\left(
\widetilde{P}\right)  .$ It is important to note that if $p$ is PT-symmetric
then it is obvious that
\begin{equation}
\overline{\int_{0}^{\pi}p(t)dt}=\int_{0}^{\pi}\overline{p(t)}dt=\int_{0}^{\pi
}\overline{p(-t)}dt=\int_{0}^{\pi}p(t)dt, \tag{28}%
\end{equation}
and hence $\int_{0}^{\pi}p(t)dt$ is a real number. Therefore, $\gamma$ and
$\tau$ are real numbers. Since $L_{t}\left(  \widetilde{P}\right)
=L_{t+2}\left(  \widetilde{P}\right)  ,$ (see boundary conditions (2)) and the
spectrum of $L\left(  \widetilde{P}\right)  $ is the union of the spectra of
$L_{t}\left(  \widetilde{P}\right)  $ for $t\in\lbrack0,2),$ we have
\[
\sigma\left(  L\left(  \widetilde{P}\right)  \right)  =%
{\textstyle\bigcup\limits_{t\in\lbrack0,2)}}
\sigma\left(  L_{t}\left(  \widetilde{P}\right)  \right)  =%
{\textstyle\bigcup\limits_{t\in\lbrack0,2)}}
\sigma\left(  L_{t-\tau}\left(  \widetilde{P}\right)  \right)  =%
{\textstyle\bigcup\limits_{t\in\lbrack0,2)}}
\left(  \gamma+\sigma\left(  L_{t}\left(  P\right)  \right)  \right)
=\gamma+\sigma\left(  L\left(  P\right)  \right)  .
\]
Therefore the following statement is a consequence of Theorem 4.

\begin{theorem}
For $\left\vert k\right\vert >N$, the intervals $(2k+\varepsilon_{k}%
+\gamma,2k+1-\varepsilon_{k}+\gamma)$ and $(2k+\varepsilon_{k}+\gamma
,2k+2-\varepsilon_{k}+\gamma)$ belong to the spectrum of the operator
$L\left(  \widetilde{P}\right)  $ and (25) continues to hold if $\widetilde
{P}$ is replaced by $P,$ where $N$ is defined in Theorem 3.
\end{theorem}

\section{Concluding Remarks}

First, let us discuss general PT-symmetric operators. The definition of
PT-symmetry for a general operator $A$ defined in some subspace $D(A)$ of a
function space $H$ is given by the relation
\begin{equation}
PT(Lf)=L(PTf) \tag{29}%
\end{equation}
for all $f$ $\in D(A)$ (see, for example [2]). We use this definition for
$L(n,m)$ (see Definition 1). To use (29) we should assume that $PTf\in D(A).$
Thus, it is natural to use the following definition.

\begin{definition}
We say that the operator $A$ is PT-symmetric if $PTf\in D(A)$ whenever $f\in
D(A),$ and
\begin{equation}
PTAf=APTf \tag{30}%
\end{equation}
for all $f\in D(A).$
\end{definition}

The domain of the differential operator $L(n,m,U)$ generated in the space
$L_{2}^{m}(a,b)$ by the differential expression (3) and the boundary
conditions
\begin{equation}
U_{1}\mathbf{y}=0,\text{ }U_{2}\mathbf{y}=0,...,U_{n}\mathbf{y}=0 \tag{31}%
\end{equation}
is the set of functions $\mathbf{f\in}\left(  W_{1}^{n}(a,b)\right)  ^{m}$
such that $l(n,m)\mathbf{f\in}L_{2}^{m}(a,b)$ and $U_{j}\mathbf{f}=0$ for
$j=1,2,...,n$. Therefore, in order to consider the PT-symmetry of the operator
$L(n,m,U)$ we have to determine whether $PT\mathbf{f}$ satisfy the same
boundary condition as $\mathbf{f}$, since $PT\mathbf{f}$ belongs $D(L(n,m,U))$
if and only if it satisfies the corresponding boundary conditions. As an
example, let us consider the boundary condition (2). If $\mathbf{y}$ satisfies
(2), then $e^{-itx}\mathbf{y}$ is a periodic function. Therefore, the domain
$D(L_{t}(P))$ of $L_{t}(P)$ can be defined as the set of functions of the form
$e^{itx}\mathbf{f,}$ where $\mathbf{f}$ is a periodic function. If $t$ is a
complex number, then
\[
PTe^{itx}\mathbf{f}(x)=P(e^{-i\overline{t}x}\overline{\mathbf{f}%
(x)})=e^{i\overline{t}x}\overline{\mathbf{f}(-x)}.
\]
Therefore, $PTe^{itx}\mathbf{f}(x)\in D(L_{t}(P))$ if and only if $t$ is real.
In this sense, $L_{t}(P)$ is PT-symmetric operator only for real $t$. It
follows from (26) that if
\begin{equation}
\operatorname{Im}\left(  \int_{0}^{\pi}\left(  p_{1}(t)-p_{4}(t)\right)
dt\right)  \neq0, \tag{32}%
\end{equation}
then $L_{t-\tau}\left(  \widetilde{P}\right)  $ is not a PT-symmetric
operator, since $\operatorname{Im}(t-\tau)\neq0$. Fortunately, the mean value
of the PT-symmetric function is a real number (see (28)). Therefore, if
$p_{1}$ and $p_{4}$ are PT-symmetric function, then $\operatorname{Im}%
(t-\tau)=0$ and $L_{t-\tau}\left(  \widetilde{P}\right)  $ is a PT-symmetric operator.

Thus, for the differential operator $L(n,m,U)$ generated in the space
$L_{2}^{m}(a,b)$ by the differential expression (3) and boundary condition
(31), the assumption that $p_{k,i,j}$ are PT-symmetric functions for all
$k,i$, $j$ does not imply that $L(n,m,U)$ is a PT-symmetric operator. We also
need to consider boundary condition (31). It follows from Definition 2 that

\begin{proposition}
The operator $L(n,m,U)$ is PT-symmetric if and only if $l(n,m)$ is a
PT-symmetric expression and $PT\mathbf{f}$ satisfies (31) for all $\mathbf{f}$
satisfying (31).
\end{proposition}

Arguing as in the proof of Theorem 2, we obtain

\begin{proposition}
If $L(n,m,U)$ is a PT-symmetric operator and $\lambda$ is an eigenvalue of
$L(n,m,U)$, then $\overline{\lambda}$ is also an eigenvalue of $L(n,m,U).$
\end{proposition}

\begin{proof}
$(a)$ Let $\lambda$ be an eigenvalue of the PT-symmetric operator $L(n,m,U).$
Then
\[
L(n,m,U)\mathbf{f}=\lambda\mathbf{f}%
\]
for some $\mathbf{f}\in D(L(n,m,U))$. Applying the operator $PT$ to both sides
of this equality and then using Definition 2, we obtain
\[
PTL(n,m,U)\mathbf{f}(x)=\overline{\lambda}\overline{\mathbf{f}(-x)},\text{ }%
\]
and $\overline{\mathbf{f}(-x)}\in D(L(n,m,U)).$ Hence,
\[
L(n,m,U)\overline{\mathbf{f}(-x)}=L(n,m,U)PT\mathbf{f}=PTL(n,m,U)\mathbf{f}%
(x)=\overline{\lambda}\overline{\mathbf{f}(-x)}.
\]
Thus, $\overline{\lambda}$ is also an eigenvalue of $L(n,m,U).$
\end{proof}

Now let us consider the real eigenvalues of $L(n,m,U)$ by using Proposition 2
and the arguments from the proof of Theorem 3.

\begin{proposition}
If there exist $a\in\mathbb{R}$ and $\varepsilon>0$ such that the circle
$\left\{  z\in\mathbb{C}:\left\vert z-a\right\vert =\varepsilon\right\}  $
contains only one eigenvalue $\lambda$ of the PT-symmetric operator $L(n,m,U)$
in its interior, then this eigenvalue is real.
\end{proposition}

\begin{proof}
If $\lambda$ is not real, then, by Proposition 2, $\overline{\lambda}$ is also
an eigenvalue lying inside $\left\{  z\in\mathbb{C}:\left\vert z-a\right\vert
=\varepsilon\right\}  $ which contradicts the assumption of the proposition.
\end{proof}

Note that in Proposition 3 we do not assume that $\lambda$ is a simple
eigenvalue; it may be a multiple eigenvalue. Let us discuss some applications
of this proposition for the PT-symmetric operator $L(P,U)$ generated in the
space $L_{2}^{2}(a,b)$ by the differential expression (1) and some boundary
conditions $U(y)=0.$ Let $\lambda_{1}(P),\lambda_{2}(P),...$ be eigenvalues of
$L(P,U)$ such that
\begin{equation}
\left\vert \lambda_{1}(P)\right\vert \leq\lambda_{2}(P)\leq.... \tag{33}%
\end{equation}

\begin{proposition}
If the eigenvalue $\lambda_{k}(O)$ of $L(O,U)$, is real and
\begin{equation}
\left\vert \lambda_{k}(P)-\lambda_{k}(O)\right\vert <\frac{1}{2}d_{k}(P),
\tag{34}%
\end{equation}
then the eigenvalues $\lambda_{k}(P)$ of $L(P,U)$ is also real, where $O$ is
the zero matrix and
\[
d_{k}(P)=\min_{n\neq k}\left\vert \lambda_{k}(P)-\lambda_{n}(P)\right\vert .
\]

\end{proposition}

\begin{proof}
If $\lambda_{k}(P)=\lambda_{k}(O),$ then $\lambda_{k}(P)$ is real, since
$\lambda_{k}(O)$ is real. It remains to consider the case%
\[
\left\vert \lambda_{k}(P)-\lambda_{k}(O)\right\vert >\varepsilon
\]
for some $\varepsilon>0.$ Then, by (34) $d_{k}(P)>2\varepsilon,$ and hence by
(34), $\lambda_{k}(P)$ is a simple eigenvalue. Moreover, the circle $\left\{
z\in\mathbb{C}:\left\vert z-\lambda_{k}(O)\right\vert =\frac{1}{2}%
d_{k}(P)\right\}  $ contains only one eigenvalue $\lambda_{k}(P)$ of the
PT-symmetric operator $L(P,U)$ in its interior. Therefore the result follows
from Proposition 3.
\end{proof}

This proposition can, for example, be applied to the operator $L(P,U),$ if
$U(y)=0$ is a strongly regular boundary condition and the eigenvalues of
$L(O,U)$ are real. In this case, in general, there exists asymptotic formulas
\[
\lambda_{k}(P)=\lambda_{k}(O)+\alpha_{k}(P)
\]
as $k\rightarrow\infty,$ where $\alpha_{k}(P)<\frac{1}{2}d_{k}(P)$.

\end{document}